\documentclass[12pt,a4paper, reqno]{amsart}
\usepackage{amsfonts,amsmath,amsthm,amssymb, amscd}
\usepackage[top=2cm, bottom=5cm, left=1cm, right=1cm]{geometry}

\usepackage[foot]{amsaddr}

\newtheorem{theorem}{Theorem}

\newtheorem{corollary}{Corollary}

\newtheorem{definition}{Definition}
\newtheorem{example}{Example}

\newtheorem{remark}{Remark}

\makeatletter
\renewcommand{\email}[2][]{%
  \ifx\emails\@empty\relax\else{\g@addto@macro\emails{,\space}}\fi%
  \@ifnotempty{#1}{\g@addto@macro\emails{\textrm{(#1)}\space}}%
  \g@addto@macro\emails{#2}%
}
\makeatother

\title{On a generalization of the Pentagonal Number Theorem}%

\thanks{
Supported by the UAEU grants: StartUp Grant 2016: G00002235}

\author{Ho-Hon Leung}
\address{Department of Mathematical Sciences, United Arab Emirates University, Al Ain, 15551, United Arab Emirates}
\email{hohon.leung@uaeu.ac.ae}

\begin{document}

\begin{abstract}
We study a generalization of the classical Pentagonal Number Theorem and its applications. We derive new identities for certain infinite series, recurrence relations and convolution sums for certain restricted partitions and divisor sums. We also derive new identities for Bell polynomials. 
\end{abstract}

\maketitle

\section{Introduction} \label{section1}

The {\it Pentagonal Number Theorem} is one of Euler's most profound discoveries. It is the following identity:
\[ \prod_{n=1}^\infty (1-q^n) = \sum_{n=-\infty}^\infty (-1)^n q^{n(3n-1)/2}
\]where $n(3n-1)/2$ is called the $n^\text{th}$ pentagonal numbers. The pentagonal numbers represent the number of distinct points which may be arranged to form superimposed regular pentagons such that the number of points on the sides of each respective pentagonal is the same. Bell's article \cite{Bell} is an excellent reference about the Pentagonal Number Theorem and its applications from the historical perspective. Andrews's article \cite{Andrews1} is devoted to a modern exposition of Euler's original proof of the theorem.

Let $\mathbb{N}=\{1,2,\dots\}$ and $\mathbb{N}_0=\{0,1,2,\dots\}$. The function $p(n)$ is the number of integer partitions of $n$. The function $\sigma(n)$ is the sum of all divisors of $n$. Euler applied the Pentagonal Number Theorem to study various properties of integer partitions and divisor sums. In particular, there is a convolution formula that connects the functions $p(n)$ and $\sigma(n)$:
\begin{align}
\label{euler} np(n) &=\sum_{k=0}^\infty \sigma(k) p(n-k).\end{align}Apart from this, there are recurrence relations for $p(n)$ and $\sigma(n)$ which are intimately related to the Pentagonal Number Theorem:
\[  p(n)=p(n-1)+p(n-2) -p(n-5) -p(n-7)+\dots, \]
\[ \sigma(n) =\sigma(n-1)+\sigma(n-2) - \sigma(n-5) - \sigma(n-7) +\dots.\]The readers are invited to read the survey article by Osler et. al.\cite{Osler} for a readable account of the connections between functions $p(n)$ and $\sigma(n)$.

One may wonder if there are recurrence relations and convolution formulas for (restricted) integer partitions and divisor sums in terms of other polygonal numbers (e.g. triangular numbers, heptagonal numbers $\dots$). The goal of this article is to give positive answers to this question based on a generalization of the Pentagonal Number Theorem.

Unless mentioned otherwise, throughout the paper, all equations in the variable $q$ which involve infinite sums and infinite products hold true if $|q|<1$.

\section{Main Results} \label{section2}

\subsection{Main theorem and some corollaries} \label{section2.1}

Let $g>3$ and $n\in\mathbb{Z}$. We denote the generalized $n^{\text{th}}$ $g$-gonal number (\cite[p.40]{Cheung}) by \[P_{g,n}=\frac{n((g-2)n-(g-4))}{2}.\]For $n\in\mathbb{N}$, $P_{g,n}$ represents the number of points which may be arranged to form regular $g$-gons such that the number of points on the sides of each respective $g$-gon is the same.

Let $(a;q)_n$ be the $q$-Pochhammer symbol for $n\geq 1$. That is, 
\[ (a;q)_n:=\prod_{k=0}^{n-1} (1-aq^k)=(1-a)(1-aq)(1-aq^2)\dots (1-aq^{n-1}).\]Considered as a formal power series in $q$, the definition of $q$-Pochhammer symbol can be extended to an infinite product. That is, \[(a;q)_\infty := \prod_{k=0}^\infty (1-aq^k).\]We note that $(q;q)_\infty$ is the Euler's function.

We state the following theorem.

\begin{theorem} \label{theorem1}Let $g\geq 5$.
\[   (q; q^{g-2})_\infty (q^{g-3}; q^{g-2})_\infty(q^{g-2}; q^{g-2})_\infty=\sum_{n=-\infty}^\infty (-1)^n  q^{P_{g,n}}.  \]
\end{theorem}

\begin{proof}
We replace $q$ by $q^{(g-2)/2}$ and $z$ by $q^{-(g-4)/2}$ in the Jacobi's Triple Product Identity, \[\sum_{n=-\infty}^\infty q^{n^2} z^n = \prod_{n=1}^\infty (1-q^{2n})(1+q^{2n-1}z)(1+q^{2n-1}z^{-1})\]to get the desired result.
\end{proof}

\begin{remark}
If $g=5$, we get Euler's Pentagonal Number Theorem by Theorem \ref{theorem1}.
\end{remark}

If $n\in\mathbb{N}$, then \[\frac{(-n)((g-2)(-n)-(g-4))}{2}=\frac{n((g-2)n+(g-4))}{2}.\]Let $g>3$ and $n\in\mathbb{N}$. Let $Q_{g,n}$ be the following numbers:\[Q_{g,n}=\frac{n((g-2)n+(g-4))}{2}.\]We note that $Q_{g,n}=P_{g,-n}$. Theorem \ref{theorem1} can be restated as 
\begin{align}
\label{equation1} (q; q^{g-2})_\infty (q^{g-3}; q^{g-2})_\infty(q^{g-2}; q^{g-2})_\infty&=1+ \sum_{n=1}^\infty  (-1)^n (q^{P_{g,n}} + q^{Q_{g,n}}).
\end{align}

\noindent The $n^{\text{th}}$ triangular number $\Delta_n$ is \[\Delta_n =\frac{n(n+1)}{2}=P_{3,n}.\]We note that \[ \Delta_{2n-1}=P_{6,n}, \quad \Delta_{2n} =Q_{6,n}. \]The following corollary is clear by (\ref{equation1}) when $g=6$.

\begin{corollary} \label{corollary0.1}
\[
(q; q^{4})_\infty (q^{3}; q^{4})_\infty(q^{4}; q^{4})_\infty= 1- \sum_{k=0}^\infty  \big(q^{\Delta_{4k+1}}+q^{\Delta_{4k+2}}\big) +\sum_{k=0}^\infty  \big(q^{\Delta_{4k+3}}+q^{\Delta_{4k+4}}\big).
\]
\end{corollary}

\begin{corollary} \label{corollary0.2}
\[ 
\sum_{n=0}^\infty \frac{q^{n^2+n}}{(q;q)_n}=\Big( \sum_{n=0}^\infty p(n) q^n  \Big) \Big( 1+ \sum_{n=1}^\infty  (-1)^n (q^{P_{7,n}} + q^{Q_{7,n}})\Big).
\]
\end{corollary}

\begin{proof}
Based on one of the Rogers-Ramanujan identities (\cite{Ramanujan}, \cite{Rogers}), we get  
\begin{align}
\nonumber  \sum_{n=0}^\infty \frac{q^{n^2+n}}{(q;q)_n}  &= \frac{1}{(q^2; q^5)_\infty (q^3 ; q^5)_\infty}=\frac{(q;q^5)_\infty (q^4;q^5)_\infty (q^5;q^5)_\infty}{(q;q^5)_\infty (q^2;q^5)_\infty (q^3;q^5)_\infty (q^4; q^5)_\infty (q^5;q^5)_\infty}\\
\label{RR}  &=  \Big( \frac{1}{(q;q)_\infty}\Big) \Big( (q;q^5)_\infty (q^4;q^5)_\infty (q^5;q^5)_\infty \Big).
\end{align}We get the desired result based on the generating function of $p(n)$ and the identity (\ref{equation1}) when $g=7$.
\end{proof}

\begin{corollary} \label{corollary0.3}
\[
     1+ \sum_{n=1}^\infty  (-1)^n (q^{5P_{5,n}} + q^{5Q_{5,n}})=\Big( \sum_{n=0}^\infty \frac{q^{n^2}}{(q;q)_n} \Big)\Big( 1+ \sum_{n=1}^\infty  (-1)^n (q^{P_{7,n}} + q^{Q_{7,n}})\Big).
\]
\end{corollary}

\begin{proof}
We replace $q$ by $q^5$ in the Pentagonal Number Theorem to get  
\begin{align}
\label{PNT}  (q^5;q^5)_\infty &= 1+ \sum_{n=1}^\infty  (-1)^n (q^{5P_{5,n}} + q^{5Q_{5,n}}).
\end{align}On the other hand,
\begin{align}
\label{equation0.01}  (q^5;q^5)_\infty &= \Big( \frac{1}{(q;q)_\infty (q^4; q^5)_\infty}\Big) \Big( (q;q^5)_\infty (q^4;q^5)_\infty (q^5;q^5)_\infty \Big).
\end{align}We recall one of the Rogers-Ramanujan identities  (\cite{Ramanujan}, \cite{Rogers}),
\begin{align}
\label{equation0.02}  \sum_{n=0}^\infty \frac{q^{n^2}}{(q;q)_n}    &=  \frac{1}{(q;q)_\infty (q^4; q^5)_\infty}.
\end{align}By putting (\ref{PNT}), (\ref{equation0.02}) into (\ref{equation0.01}), we get the desired result based on the identity (\ref{equation1}) when $g=7$.
\end{proof}

\subsection{Recurrence relations for some restricted integer partitions} \label{section2.2}

\begin{definition}
The function $(p_{2,5}+p_{3,5})(n)$ is the number of partitions of $n$ such that each part is either congruent to $2$ modulo $5$ or $3$ modulo $5$. We extend the domain of $(p_{2,5}+p_{3,5})(n)$ to $\mathbb{Z}$ by setting $(p_{2,5}+p_{3,5})(x)=0$ if $x\notin \mathbb{N}_0$.
\end{definition}

\begin{theorem} \label{corollary0.25}
\[
(p_{2,5}+p_{3,5})(n)=p(n)+\sum_{k=1}^\infty (-1)^k (p(n-P_{7,k}) +p(n-Q_{7,k})).\]
\end{theorem}

\begin{proof}
By Corollary \ref{corollary0.2} and the generating function of $(p_{2,5}+p_{3,5})(n)$, we get 
\begin{align*}
\sum_{n=0}^\infty (p_{2,5}+p_{3,5})(n) q^n &=\frac{1}{(q^2; q^5)_\infty (q^3 ; q^5)_\infty} \\ 
&= \Big( \sum_{n=0}^\infty p(n) q^n  \Big) \Big( 1+ \sum_{n=1}^\infty  (-1)^n (q^{P_{7,n}} + q^{Q_{7,n}})\Big).
\end{align*}We get the desired result by comparing coefficients of $q^n$ on both sides of the equation.
\end{proof}

\begin{example}
Let $n=8$. Then we have $8=2+2+2+2=3+3+2$. So, $(p_{2,5}+p_{3,5})(8)=3$. The generalized heptagonal numbers are $P_{7,1}=1, Q_{7,1}=4,  P_{7,2}=7 \dots$ and hence \begin{align*}
p(8)-p(8-P_{7,1})-p(8-Q_{7,1})+p(8-P_{7,2})&=p(8)-p(7)-p(4)+p(1)\\&=22-15-5+1=3\\&=(p_{2,5}+p_{3,5})(8).\end{align*}

\end{example}

\begin{definition}
The function $(p_{1,5}+p_{4,5})(n)$ is the number of partitions of $n$ such that each part is either congruent to $1$ modulo $5$ or $4$ modulo $5$. We extend the domain of $(p_{1,5}+p_{4,5})(n)$ to $\mathbb{Z}$ by setting $(p_{1,5}+p_{4,5})(x)=0$ if $x\notin \mathbb{N}_0$.
\end{definition}

\begin{theorem} \label{corollary0.35}
\begin{align*}
  (p_{1,5}+p_{4,5})(n) &=
    M+\sum_{k=1}^\infty (-1)^{k+1}\big( (p_{1,5}+p_{4,5})(n-P_{7,k})+(p_{1,5}+p_{4,5})(n-Q_{7,k}))\big). 
\end{align*}where 
\begin{align*}
  M &=
\begin{cases}
    (-1)^m, &n=5P_{5,m}\text{ or }5Q_{5,m}\text{ for some }m; \\
    0,              & \text{otherwise.}
\end{cases}
\end{align*}
\end{theorem}

\begin{proof}
The generating function of $(p_{1,5}+p_{4,5})(n)$ is \[ \sum_{n=0}^\infty (p_{1,5}+p_{4,5})(n) q^n=\frac{1}{(q; q^5)_\infty (q^4 ; q^5)_\infty}.\]
Now the result is obvious by comparing coefficients of $q^n$ in (\ref{equation0.01}) and by Corollary \ref{corollary0.3}.
\end{proof}

\begin{example}
Let $n=9$. The generalized heptagonal numbers are $P_{7,1}=1, Q_{7,1}=4, P_{7,2}=7 \dots$. By Theorem \ref{corollary0.35}, we get \[(p_{1,5}+p_{4,5})(9)=(p_{1,5}+p_{4,5})(8)+(p_{1,5}+p_{4,5})(5)-(p_{1,5}+p_{4,5})(2). \]It can be easily verified since 
\begin{align*}
(p_{1,5}+p_{4,5})(9)&=4,\quad 9=1+1+1+1+1+1+1+1+1=4+1+1+1+1+1=4+4+1,\\
(p_{1,5}+p_{4,5})(8)&=3, \quad 1+1+1+1+1+1+1+1=4+1+1+1+1=4+4,\\
(p_{1,5}+p_{4,5})(5)&=2, \quad 1+1+1+1+1=4+1,\\
(p_{1,5}+p_{4,5})(2)&=1, \quad 1+1.
\end{align*}
\end{example}

\begin{definition}
The function $q(n)$ is the number of partitions of $n$ such that each part is distinct. We extend the domain of $q(n)$ to $\mathbb{Q}$ by setting $q(x)=0$ if $x\notin \mathbb{N}_0$.
\end{definition}

\noindent We obtain two recurrence relations for $q(n)$.

\begin{theorem} \label{lemma1}
\[
q(n) = K+\sum_{k=1}^\infty (-1)^{k+1} \big( q(n-2P_{5,k}) +q(n-2Q_{5,k})  \big)
\]where 
\begin{align*}
  K &=
\begin{cases}
    1, &n=\Delta_m\text{ for some }m; \\
    0,              & \text{otherwise.}
\end{cases}
\end{align*}
\end{theorem}

\begin{proof}
By an identity due to Guass (\cite[p.40]{Cheung}), 
\begin{align}
\label{c1} \sum_{n=0}^\infty q^{\Delta_n} &= \frac{(q^2;q^2)_\infty}{(q;q^2)_\infty} =\Big((q^2;q^2)_\infty )\Big)\cdot \Big( \frac{1}{(q;q^2)_\infty} \Big)=\Big( (q^2;q^2)_\infty )   \Big)\cdot \Big( \sum_{n=1}^\infty q(n) q^n  \Big)
\end{align}where the last equality is due to Euler's Theorem. We replace $q$ by $q^2$ in the Pentagonal Number Theorem to get  
\begin{align}
\label{PNT2}  (q^2;q^2)_\infty &= 1+ \sum_{n=1}^\infty  (-1)^n (q^{2P_{5,n}} + q^{2Q_{5,n}}).
\end{align}By putting (\ref{PNT2}) into (\ref{c1}) and comparing coefficients of $q^n$ on both sides of the equation, we get the result as desired.
\end{proof}

\begin{theorem} \label{lemma2}
\[
q(n) = L+\sum_{k=1}^\infty (-1)^{k+1} \Big( q\big( \frac{2n-P_{5,k}}{2}\big) +q\big( \frac{2n-Q_{5,k}}{2}\big)  \Big)
\]where 
\begin{align*}
  L &=
\begin{cases}
    -1, &2n=\Delta_{4k+1}, \Delta_{4k+2}\text{ for some }k; \\
    1, &2n=\Delta_{4k+3}, \Delta_{4k+4}\text{ for some }k; \\
    0,              & \text{otherwise.}
\end{cases}
\end{align*}
\end{theorem}

\begin{proof}
\begin{align}
\label{c2} (q;q^4)_\infty (q^3;q^4)_\infty (q^4;q^4)_\infty &= \Big( (q;q)_\infty \Big) \Big( \frac{1}{(q^2;q^4)_\infty} \Big).
\end{align}We replace $q$ by $q^2$ in Euler's Theorem to get 
\begin{align}
\label{c3} \Big( \frac{1}{(q^2;q^4)_\infty} \Big) &= \prod_{k=1}^\infty (1+q^{2k}) =\sum_{n=0}^\infty q(n) q^{2n}.
\end{align} By putting (\ref{c3}) into (\ref{c2}), applying Corollary \ref{corollary0.1} to the left hand side of (\ref{c2}) and applying the Pentagonal Number Theorem to $(q;q)_\infty$, we get
\begin{align}
\label{c4} 1- \sum_{k=0}^\infty  \big(q^{\Delta_{4k+1}}+q^{\Delta_{4k+2}}\big) +\sum_{k=0}^\infty  \big(q^{\Delta_{4k+3}}+q^{\Delta_{4k+4}}\big) &= \Big(1+ \sum_{n=1}^\infty  (-1)^n (q^{P_{5,n}} + q^{Q_{5,n}})\Big)\cdot \Big( \sum_{n=1}^\infty q(n) q^{2n}\Big).
\end{align}Comparing the coefficients of $q^n$ on both sides of (\ref{c4}), we get the result as desired.
\end{proof}

\begin{example}
Let $n=15$. We note that $15=\Delta_{5}$. The generalized pentagonal numbers are \[P_{5,1}=1, \text{ }Q_{5,1}=2,\text{ } P_{5,2}=5,\text{ } Q_{5,2}=7,\text{ } P_{5,3}=12, \text{ } Q_{5,3}=15,\text{ } P_{5,4}=22,\text{ } Q_{5,4}=26\dots.\]By Theorem \ref{lemma1}, 
\begin{align*}
 q(15)&= 1+( q(15-2P_{5,1}) +q(15-2Q_{5,1})    ) -(q(15-2P_{5,2}) +q(15-2Q_{5,2})) \\
   &= 1+q(13)+q(11)-q(5) -q(1)=1+18+12-3-1=27.
\end{align*}Alternatively, by Theorem \ref{lemma2} and the fact that $2(15)=30\neq \Delta_{m} $ for any $m\in\mathbb{N}$.
\begin{align*}
 q(15) &= 0+q\Big(\frac{30-2}{2}\Big)+q\Big(\frac{30-12}{2}\Big)-q\Big(\frac{30-22}{2}\Big) -q\Big(\frac{30-26}{2}\Big)\\
  &= q(14)+q(9)-q(4)-q(2)=22+8-2-1=27.
\end{align*}
\end{example}

\begin{definition} 
Let $n, r\in\mathbb{N}_0$. Let $m\in \mathbb{N}$. The function $p_{r,m}(n)$ is defined to be the number of partitions of $n$ such that each part is congruent to $r$ modulo $m$. We extend the domain of $p_{r,m}(n)$ to $\mathbb{Z}$ by setting $p_{r,m}(x)=0$ if $x\notin \mathbb{N}_0$.
\end{definition}

\begin{definition}
Let $n\in\mathbb{Z}$ and $m\in\mathbb{N}$. The restricted integer partition $p'_m (n)$ is defined by \[p'_m(n):=p_{m-1,m}(n)+p_{0,m}(n)+p_{1,m}(n)\]where $p'_m(0):=1$ and $p'_m(n)=0$ for $n\notin\mathbb{N}_0$.
\end{definition}

\noindent In particular, if $m=3$, then $p'_m(n)=p'_3(n)=p(n)$.

The generating function of $p'_m(n)$ is as follows: 
\begin{align}
\label{equation0.5} \sum_{n=0}^\infty p'_m(n) q^n &=\frac{1}{(q;q^m)_\infty (q^{m-1};q^m)_\infty (q^m; q^m)_\infty}.
\end{align}

\noindent It is convenient to introduce the number $e_{g,n}$ given by 
\begin{align*}
  e_{g,n} &=
\begin{cases}
    1, &n=0; \\
    (-1)^k,& \text{if }  n=P_{g,k}\text{ or }Q_{g,k}; \\
    0,              & \text{otherwise.}
\end{cases}
\end{align*}
Theorem \ref{theorem1} can be rewritten as 
\begin{align}
\label{equation2} (q; q^{g-2})_\infty (q^{g-3}; q^{g-2})_\infty(q^{g-2}; q^{g-2})_\infty&=\sum_{n=0}^\infty e_{g,n} q^n.
\end{align}

\noindent We obtain the following recurrence relation for the restricted partition $p'_m(n)$ if $m\geq 3$.

\begin{theorem} \label{corollary1}
Let $m\geq 3$  and $n\in \mathbb{N}$.
\begin{align*}
  p'_{m}(n)=\sum_{k=1}^\infty (-1)^{k+1} \big(p'_{m}(n- P_{m+2,k})+p'_m(n-Q_{m+2,k})\big). 
\end{align*}
\end{theorem}

\begin{proof}
By (\ref{equation0.5}) and (\ref{equation2}), 
\begin{align*}
\Big(\sum_{n=0}^\infty p'_m(n) q^n \Big)\Big( \sum_{n=0}^\infty e_{m+2,n} q^n  \Big) &= 1
\end{align*}By comparing coefficients of $q^n$ for $n\geq 1$ on both sides of the equation, we get 
\begin{align*}
  p'_m(n) + \sum_{k=1}^\infty (-1)^k \big( p'_m(n-P_{m+2,k})+p'_m(n-Q_{m+2,k})\big) &= 0.
\end{align*}
\end{proof}

\begin{remark}
In the case $m=3$, Theorem \ref{corollary1} is reduced to the well-known recurrence relation for the integer partition $p(n)$.
\end{remark}

\begin{definition}
The function $p'_{e,m}(n)$ is the number of partitions of $n$ such that each partition has an even number of distinct parts, and each part is congruent to $0$ modulo $m$, $1$ modulo $m$ or $m-1$ modulo $m$. 
\end{definition}

\begin{definition}
The function $p'_{o,m}(n)$ is the number of partitions of $n$ such that each partition has an odd number of distinct parts, and each part is congruent to $0$ modulo $m$, $1$ modulo $m$ or $m-1$ modulo $m$. 
\end{definition}

\begin{theorem} \label{legendre}
Let $n\in\mathbb{N}$ and $m\geq 3$.
\begin{align*}
  p'_{e,m}(n) -p'_{o,m}(n) &= 
\begin{cases}
    (-1)^k, &n=P_{m+2,k}\text{ or }Q_{m+2,k}\text{ for some }k\in\mathbb{N}; \\
    0,              & \text{otherwise.}
\end{cases}
\end{align*}
\end{theorem}

\begin{proof}
In the infinite product expansion of \[(q;q^m)_\infty (q^{m-1}; q^m)_\infty (q^m ; q^m)_\infty, \]the coefficients of $q^N$ has $+1$ contribution from a partition of $N$ that consists of an even number of parts, where each part is congruent to $0$ modulo $m$, $1$ modulo $m$ or $m-1$ modulo $m$. It has $-1$ contribution from a partition of $N$ that consists of an odd number of parts, where each part is congruent to $0$ modulo $m$, $1$ modulo $m$ or $m-1$ modulo $m$. By Theorem \ref{theorem1}, our result follows immediately.
\end{proof}

\begin{remark}
In the case $m=3$, Theorem \ref{legendre} is reduced to the interesting observation made by Legendre on the Pentagonal Number Theorem (\cite[p.2]{Andrews1}).
\end{remark}

\subsection{Recurrence relations and convolution sums for restricted divisor sums} \label{section2.3}

\begin{definition}
Let $r\in\mathbb{N}_0$. Let $n,m\in \mathbb{N}$. The function $\sigma_{r,m}(n)$ is defined by \[\sigma_{r,m} (n):=\sum_{\{d\in\mathbb{N}, d|n, d\equiv r\text{ mod }m\}}d. \]
\end{definition}

\begin{definition}
Let $n,m\in \mathbb{N}$. The restricted divisor sum $\sigma'_m(n)$ is defined by \[\sigma'_m(n):=\sigma_{m-1,m}(n)+\sigma_{0,m}(n)+\sigma_{1,m}(n).\]
\end{definition}

\noindent The restricted Lambert series for $\sigma_{r,m}(n)$ is \[\sum_{n=1}^\infty \sigma_{r,m}(n)q^n=\sum_{\{n\text{ }|\text{ }n\equiv r\text{ mod }m\}}\frac{nq^n}{1-q^n}.\]The restricted Lambert series for $\sigma'_m(n)$ is
\begin{align}
\label{equation3} \sum_{n=1}^\infty \sigma'_m(n)q^n & =\sum_{\{n\text{ }|\text{ }n\equiv 0\text{ mod }m,\text{ }n\equiv 1\text{ mod }m,\text{ }n\equiv m-1\text{ mod }m\}}\frac{nq^n}{1-q^n}.
\end{align}

\noindent The following theorems connect $p'_m(n)$, $\sigma'_m(n)$ and the generalized $n^{\text{th}}$-gonal numbers.

\begin{theorem} \label{theorem2} 
Let $m\geq 3$  and $n\in \mathbb{N}$.
\[
\sigma'_m(n) =\sum_{k=0}^\infty (-1)^{k+1} ( P_{m+2,k}  \cdot p'_m(n-P_{m+2,k})+ Q_{m+2,k} \cdot p'_m(n-Q_{m+2,k})).
\]
\end{theorem}

\begin{proof}
Let $m\in\mathbb{N}$ such that $m\geq 3$. Define the function $G_m(x)$ as follows: \begin{align}
\label{G} G_m(x)&=\prod_{k=1}^\infty (1-x^{km})(1-x^{km-(m-1)})(1-x^{km-1}).\end{align}Taking the logarithm of $G_m(x)$, we get \[\ln(G_m(x))=\sum_{k=1}^\infty \big(\ln(1-x^{km})+\ln(1-x^{km-(m-1)})+\ln(1-x^{km-1})   \big).\]Differentiating and then multiplying by $x$, we get 
\begin{align}
\label{equation4} -\frac{xG'_m(x)}{G_m(x)}=\sum_{k=1}^\infty \Big( \frac{kmx^{km}}{1-x^{km}}+\frac{(km-(m-1))x^{km-(m-1)}}{1-x^{km-(m-1)}}+\frac{(km-1)x^{km-1}}{1-x^{km-1}}   \Big).
\end{align}By (\ref{equation0.5}), 
\begin{align}
\label{equation5} \frac{1}{G_m(x)}&=\sum_{n=0}^\infty p'_m(n) x^n.
\end{align}By (\ref{equation2}),
\begin{align}
\label{equation6} xG'_m(x) &= \sum_{n=0}^\infty n e_{m+2,n} x^{n}.
\end{align}Putting (\ref{equation3}), (\ref{equation5}), (\ref{equation6}) into (\ref{equation4}), we get
\begin{align}
\nonumber \Big(\sum_{n=0}^\infty p'_m(n) x^n\Big) \Big(\sum_{n=0}^\infty n e_{m+2,n} x^{n} \Big) &= - \sum_{n=1}^\infty \sigma'_m(n)x^n.
\end{align}We get the desired result by comparing coefficients of $x^n$ on both sides of the equation.
\end{proof}

\begin{theorem} \label{theorem3} Let $m\geq 3$  and $n\in \mathbb{N}$.
\[
\sigma'_m(n) = -n e_{m+2,n}+\sum_{k=0}^{n-1} (-1)^{k+1} \big( \sigma'_m(n-P_{m+2,k})+\sigma'_m(n- Q_{m+2,k}) \big). 
\]
\end{theorem}

\begin{proof}
By (\ref{equation3}), (\ref{equation4}), (\ref{equation6}), we get 
\begin{align}
\label{equation7} -\sum_{n=0}^\infty n e_{m+2,n} x^n &= \Big( \sum_{n=0}^\infty e_{m+2,n} x^n \Big) \Big( \sum_{n=0}^\infty \sigma'_m(n) x^n \Big).
\end{align}We get the desired result by comparing coefficients of $x^n$ on both sides of the equation.
\end{proof}

\begin{theorem} \label{theorem4} Let $m\geq 3$ and $n\in \mathbb{N}$.
\[
np_m'(n) =\sum_{k=0}^\infty \sigma'_m (k) p'_m(n-k).
\]
\end{theorem}

\begin{proof}
Based on the definition of $G_m(x)$ in (\ref{G}), let the function $F_m(x)$ be \[F_m(x)= \frac{1}{G_m(x)}.\]Then
\begin{align}
\label{equation8} F_m'(x)&=-\frac{G_m'(x)}{(G_m(x))^2}=-\frac{G_m'(x)}{G_m(x)}\cdot  \frac{1}{G_m(x)}=-\frac{G_m'(x)}{G_m(x)}\cdot      F_m(x).
\end{align}By (\ref{equation3}), (\ref{equation4}), 
\begin{align}
\label{equation9} -\frac{G_m'(x)}{G_m(x)} &=\sum_{n=1}^\infty \sigma'_m(n) x^{n-1}.
\end{align}By (\ref{equation5}), 
\begin{align}
\label{equation10} F_m'(x) &= \sum_{n=0}^\infty np_m'(n) x^{n-1}.
\end{align}Putting (\ref{equation5}), (\ref{equation9}), (\ref{equation10}) into (\ref{equation8}) and multiplying by $x$, we get
\begin{align}
\label{equation11}    \sum_{n=0}^\infty  np_m'(n)x^n &= \Big(\sum_{n=0}^\infty \sigma_m'(n)x^n  \Big)\Big( \sum_{n=0}^\infty p_m'(n) x^n  \Big).
\end{align}We get the desired result by comparing coefficients of $x^n$ on both sides of the equation (\ref{equation11}).
\end{proof}

\begin{remark}
In the case $m=3$, Theorem \ref{theorem4} is identical to (\ref{euler}).
\end{remark} 

\section{Identities for Bell polynomials} \label{section3}

\subsection{Preliminaries} \label{subsection3.1}

Let $(x_1, x_2, \dots)$ be a sequence of real numbers. The {\it partial} exponential Bell polynomials are polynomials given by 
\begin{align*}
B_{n,k}(x_1,x_2,\dots,x_{n-k+1}) &= \sum_{\pi(n,k)} \frac{n!}{j_1 ! j_2 ! \dots j_{n-k+1}!} \Big(\frac{x_1}{1!} \Big)^{j_1} \Big(\frac{x_2}{2!} \Big)^{j_2}\dots \Big(\frac{x_{n-k+1}}{(n-k+1)!} \Big)^{j_{n-k+1}}
\end{align*}where $\pi (n,k)$ is the positive integer sequence $(j_1, j_2, j_{n-k+1})$ satisfies the following equations: 
\begin{align*}
j_1 +j_2 +\dots +j_{n-k+1} &=k, \\
j_1 + 2 j_2 + \dots + (n-k+1) j_{n-k+1} &=n.
\end{align*}For $n\geq 1$, the $n^{\text{th}}$-complete Bell polynomial $B_n(x_1, \dots, x_n)$ is the following:
\begin{align*}
  B_n (x_1\dots x_n)  &=\sum_{k=1}^n B_{n,k} (x_1, \dots , x_{n-k+1}).
\end{align*}The complete exponential Bell polynomials can also be defined by power series expansion as follows:
\begin{align}
  \label{equation3.1}  \text{exp}\Big( \sum_{m=1}^\infty x_m \frac{t^m}{m!} \Big) &= \sum_{n=0}^\infty B_n (x_1, \dots , x_n) \frac{t^n}{n!},
\end{align}where $B_0\equiv 1$. Alternatively, the complete Bell polyomials can be recursively defined by
\begin{align}
\label{equation3.2} B_{n+1}(x_1,\dots,x_{n+1})&=\sum_{i=0}^n \binom{n}{i} B_{n-i}(x_1,\dots, x_{n-i}) x_{i+1}.
\end{align}One interesting property of the Bell polynomials is that there exists an inversion formula in the following sense. If we define
\begin{align}
\label{equation3.25} y_n&= B_n(x_1, x_2, \dots, x_n),
\end{align}then
\begin{align}
\label{equation3.3} x_n &= \sum_{k=1}^n (-1)^{k-1} (k-1)!\cdot B_{n,k} (y_1, \dots, y_{n-k+1}).
\end{align}For detailed properties of such inverse formulas, see the paper written by Chaou et. al \cite{Chaou}. Bell polynomials were first introduced by Bell \cite{Bell}. The books written by Comtet \cite{Comtet} and Riordan \cite{Riordan} serve as excellent references for the numerous applications of Bell polynomials in combinatorics. Recently, there has been extensive research in finding identities on (partial/complete) Bell polynomials. The paper written by W. Wang and T. Wang \cite{Wang} provides many interesting identities for partial Bell polynomials. Bouroubi and Benyahia-Tani \cite{Sadek} and the author \cite{Leung} proved some new identities for complete Bell polynomials based on Ramanujan's congruences.

\subsection{Identities for complete Bell polynomials and some corollaries} \label{section3.2}

We recall the notations $\sigma'_m(n)$, $p'_m(n)$ and $e_{g,n}$ used in Section \ref{section2.2} and Section \ref{section2.3}.

\begin{theorem} \label{theorem3.1}
Let $n\in\mathbb{N}$ and $m\geq 3$.
\[
B_n(d_1,d_2,\dots, d_n)= n!\cdot e_{m+2,n}
\]where $d_n=-(n-1)!\cdot \sigma'_m(n)$.
\end{theorem}

\begin{proof}
By (\ref{G}), we take the logarithm of $G_m(x)$ and use the formal power series expansion of $\ln (1-x)$ to get
\begin{align}
\label{equation3.4} \ln(G_m(x))&=\sum_{k=1}^\infty \big(\ln(1-x^{km})+\ln(1-x^{km-(m-1)})+\ln(1-x^{km-1})   \big) \\
\nonumber  &= -\sum_{k=1}^\infty \sum_{j=1}^\infty \frac{(x^{km})^j}{j}  -\sum_{k=1}^\infty \sum_{j=1}^\infty \frac{(x^{km-(m-1)})^j}{j}  -\sum_{k=1}^\infty \sum_{j=1}^\infty \frac{(x^{km-1})^j}{j}\\
\label{equation3.5} &=  -\sum_{k=1}^\infty \frac{\sigma'_m(n)}{n} x^n = \sum_{k=1}^\infty \big(-(n-1)!\cdot \sigma'_m(n) \big) \frac{x^n}{n!}.
\end{align}Taking exponential on both sides of (\ref{equation3.5}) to get 
\begin{align}
\label{equation3.6} G_m(x) &= \exp \Big( \sum_{k=1}^\infty d_n\frac{x^n}{n!}\Big)= \sum_{n=0}^\infty B_n (d_1, \dots , d_n) \frac{x^n}{n!}
\end{align}where $d_n=-(n-1)!\cdot \sigma'_m(n)$ and the last equality is due to (\ref{equation3.1}). By (\ref{equation2}), 
\begin{align}
\label{equation3.7} G_m(x) &=\sum_{n=0}^\infty e_{m+2,n} x^n.
\end{align}Now the result is clear by comparing (\ref{equation3.6}) and (\ref{equation3.7}).
\end{proof}

\begin{theorem} \label{theorem3.2}
Let $n\in\mathbb{N}$ and $m\geq 3$.
\[
B_n(c_1,c_2,\dots, c_n)= n!\cdot p'_m(n)
\]where $c_n=(n-1)!\cdot \sigma'_m(n)$.
\end{theorem}

\begin{proof}
It is essentially the same as the proof of Theorem \ref{theorem3.1}. Let the function $F_m(x)$ be \[F_m(x)=\frac{1}{G_m(x)}.\]By taking logarithm of $F_m(x)$ and using the formal power series of $\ln (1-x)$, we get 
\begin{align}
\nonumber \ln(F_m(x))&=-\sum_{k=1}^\infty \big(\ln(1-x^{km})+\ln(1-x^{km-(m-1)})+\ln(1-x^{km-1})   \big)\\
\label{equation3.8}  &= \sum_{k=1}^\infty \big((n-1)!\cdot \sigma'_m(n) \big) \frac{x^n}{n!}.
\end{align}We get the result as desired by taking exponential on both sides of (\ref{equation3.8}), applying (\ref{equation3.1}) and generating function of $p'_m(n)$ in (\ref{equation0.5}).
\end{proof}

\begin{remark}
It is worthwhile to notice that Theorem \ref{theorem3} (resp. Theorem \ref{theorem4}) can be proved by applying Theorem \ref{theorem3.1} (resp. Theorem \ref{theorem3.2}) and the convolution properties of complete Bell polynomials shown in (\ref{equation3.2}). More precisely, by Theorem \ref{theorem3.2}, the equation (\ref{equation3.2}) becomes
\begin{align}
\nonumber (n+1)!\cdot p'_m(n+1) &= \sum_{i=0}^n \binom{n}{i} (n-i)! \cdot p'_m(n-i) \cdot i!\cdot \sigma'_m(i+1) \\
\label{special} (n+1) p'_m(n+1) &= \sum_{i=0}^n p'_m(n-i) \sigma'_m(i+1). 
\end{align}Now it is obvious that (\ref{special}) is equivalent to Theorem \ref{theorem4}. Likewise, Theorem \ref{theorem3} can be proved by using Theorem \ref{theorem3.1} and the equaion (\ref{equation3.2}). 
\end{remark}

\noindent By the inversion formulas of Bell polynomials as stated in (\ref{equation3.25}) and (\ref{equation3.3}), we immediately obtain the following corollaries due to Theorem \ref{theorem3.1} and Theorem \ref{theorem3.2} respectively.

\begin{corollary} \label{corollary3.1}
Let $n\in\mathbb{N}$ and $m\geq 3$.
\[
\sigma'_m(n) =\frac{1}{(n-1)!} \sum_{k=1}^\infty (-1)^k (k-1)!\cdot B_{n,k} (1!\cdot e_{m+2,1}, 2!\cdot e_{m+2.2}, \cdots, (n-k+1)!\cdot e_{m+2, n-k+1}). 
\]
\end{corollary}

\begin{corollary} \label{corollary3.2}
Let $n\in\mathbb{N}$ and $m\geq 3$.
\[
\sigma'_m(n) =\frac{1}{(n-1)!} \sum_{k=1}^\infty (-1)^{k-1} (k-1)!\cdot B_{n,k} (1!\cdot p'_{m}(1), 2!\cdot p'_m(2), \dots, (n-k+1)!\cdot p'_m(n-k+1)). 
\]
\end{corollary}

\noindent It might come as a surprise that RHS of the formula in Corollary \ref{corollary3.1} is equal to the RHS of the formula in Corollary \ref{corollary3.2} as the former one leads to a simple computation (many terms $e_{m+2,n}$ are zeros) while the latter one gives a rather complicated computation due to the terms $p'_m(n)$.

\section{Acknowledgements}
The author is grateful to the editor and the referees for carefully reading the paper and pointing out some mistakes in the first draft of it. Their comments were helpful to improve the quality of the article. The author is supported by Startup Grant 2016 (G00002235) from United Arab Emirates University.

\end{document}